\documentclass[12pt,letterpaper]{amsart}

\usepackage[utf8]{inputenc}
\usepackage{enumerate}%
\usepackage{amsmath,amssymb}
\usepackage{cite}
\usepackage{comment}
\usepackage{amsfonts}
\usepackage{mathrsfs}
\usepackage{xcolor}
\usepackage[colorlinks, linkcolor = blue, anchorcolor = blue, citecolor = blue]{hyperref}

\theoremstyle{definition}
\newtheorem{definition}{Definition}[section]
\newtheorem{remark}[definition]{Remark}

\newtheorem*{note}{Note}

\theoremstyle{plain}
\newtheorem{theorem}[definition]{Theorem}

\newtheorem{lemma}[definition]{Lemma}

\numberwithin{equation}{section}

\newcommand{\R}{\mathbb{R}}
\newcommand{\D}{\nabla}
\newcommand{\Om}{\Omega}
\newcommand{\ve}{\varepsilon}

\title[]{A singular profile for the relativistic heat cost\\ and the special Lagrangian curvature equation}

\author{Xiaotian Wu}
\address{School of Mathematical Sciences, Zhejiang University, Hangzhou 310058, China}
\email{xiao-tian.wu@zju.edu.cn}

\date{July 29, 2026}
\thanks{}


\begin{document}

\begin{abstract}
We study the interior regularity of generalized solutions to the Monge--Amp\`ere type
equation governing optimal transportation for the relativistic cost
$c(x,y)=\sqrt{a^2-|x-y|^2}$ on $\R^n$.
We construct an explicit one-parameter family of radially structured generalized solutions on a ball and exhibit, among them, a solution that is of class
$C^{1,\frac{1}{2n-1}}$ but of no better H\"older class: it fails to belong to
$C^{1,\beta}$ for every $\beta>\frac{1}{2n-1}$.
The construction reduces the equation to a planar autonomous system whose phase variable
$s=\dot r$ vanishes to order $(2n-1)$ in the base variable.
As an application, in dimension two we transfer the construction to the special Lagrangian
curvature equation: for every phase $\Theta\in(0,\pi/2)$ we produce a sequence of smooth
graphical solutions converging uniformly to a limit of class exactly $C^{1,1/3}$.
Consequently the two-dimensional special Lagrangian curvature equation admits no pure
interior $C^{1,\beta}$ estimate for any $\beta>\frac{1}{3}$.
\end{abstract}
\maketitle

\baselineskip16pt
\parskip3pt


\section{Introduction}

Optimal transportation seeks, given two probability measures and a cost function
$c(x,y)$ measuring the price of moving a unit of mass from $x$ to $y$, a map that
rearranges the first measure onto the second with least total cost.
Since the work of Brenier~\cite{Br91} and McCann, and the regularity theory initiated by
Caffarelli~\cite{Ca92} for the quadratic cost $c(x,y)=|x-y|^2/2$, a central theme has been
to understand how the geometry of the cost dictates the regularity of the optimal map and
of its potential.
For general costs this is governed by the Ma--Trudinger--Wang curvature tensor and the
attendant (MTW) condition introduced in~\cite{MTW05}; under (MTW) and suitable
convexity of the domains one obtains interior and global regularity
(see also Loeper~\cite{Lo09}, Trudinger--Wang~\cite{TW09}, and the
monographs~\cite{Vi09,Fi17}), whereas its failure typically permits singular optimal maps.

In this paper we study the relativistic heat cost.
Fix a positive number $a$ and define $c:\R^n\times\R^n\to\R\cup\{-\infty\}$ by
\begin{equation}\label{eq:cost}
    c(x,y):=\left\{\begin{aligned}
    \sqrt{a^2-|x-y|^2},\quad &\text{if $|x-y|<a$,}\\
    -\infty,\quad &\text{if $|x-y|\geq a$.}
    \end{aligned}
    \right.
\end{equation}

The finite cutoff at distance $a$ encodes a finiteness-of-speed constraint: only mass within
a fixed range $a$ may be transported.
Costs of this type were introduced by Brenier in his relativistic heat
equation~\cite{Br03}, where the Lagrangian $\sqrt{a^2-|v|^2}$ caps the transport velocity by
the constant $a$, and the corresponding optimal transport problem was developed by
McCann--Puel~\cite{MP09} and Bertrand--Puel~\cite{BP13}.

Let us record the equation. Let $u$ be a $C^1$ smooth $c$-convex function and $T_u$ be the $c$-normal map associated to $u$ (see Section~\ref{sec:prelim} for the precise definitions).
One checks directly from \eqref{eq:cost} that
\begin{equation}\label{eq:T_u}
    T_u(x)=x+\frac{a\D u(x)}{\sqrt{1+|\D u|^2}}.
\end{equation}
Then the prescribed Jacobian equation for $T_u$ with constant positive right-hand side $\lambda$ reads
\begin{equation}\label{eq:ot-pde}
  \det\left(
    \frac{a}{\sqrt{1+|\D u|^2}}
    \left(I-\frac{\D u\otimes\D u}{1+|\D u|^2}\right)D^2u+I
  \right)=\lambda.
\end{equation}
Equation~\eqref{eq:ot-pde} is the Monge--Amp\`ere type equation expressing that $u$ pushes
Lebesgue measure forward to $\lambda$ times Lebesgue measure under the relativistic optimal
map. When the cutoff $a$ is sent to $+\infty$ it degenerates, after rescaling, to the
classical Monge--Amp\`ere equation $\det D^2 u=\lambda$; the finite value of $a$ is exactly
what makes the operator non-uniformly elliptic near $|\D u|=\infty$ and is responsible for the
singular phenomenon below.

Our first result shows that, even for smooth right-hand side, generalized solutions
of~\eqref{eq:ot-pde} need not be better than $C^{1,\frac{1}{2n-1}}$, and that this exponent is
attained.

\begin{theorem}\label{thm:A}
    For every $n\geq2$ there is a generalized solution $u$ of \eqref{eq:ot-pde} on a ball
    $B\subset\R^n$ that belongs to $C^{1,\frac{1}{2n-1}}(B)$ but not to $C^{1,\beta}(B)$ for any
    $\beta>\frac{1}{2n-1}$.
\end{theorem}

Here a {\it generalized solution} is an Aleksandrov solution, defined through the
Monge--Amp\`ere measure $\mu_u(E)=|T_u(E)|$ (Definition~\ref{def:gs}). The exponent
$\frac{1}{2n-1}$ is sharp and tends to $0$ as $n\to\infty$, so the gradient admits no
dimension-free interior modulus of continuity. The solution is an explicit hypersurface of
revolution $u(x',t)=\sqrt{r(t)^2-|x'|^2}$ whose meridian solves an ordinary differential
equation reduced from~\eqref{eq:ot-pde}, in the spirit of Pogorelov's singular
solutions~\cite{Po78}.

It is instructive to compare Theorem~\ref{thm:A} with the positive regularity theory for
optimal transportation, since the very same H\"older exponent $\frac{1}{2n-1}$ marks the
sharp threshold there.
Assume that the cost satisfies the (MTW) condition and that the two
densities are bounded between positive constants (but possibly discontinuous).
Loeper~\cite{Lo09} proved that the optimal map is then H\"older continuous, equivalently that
the potential belongs to $C^{1,\alpha}$, with the explicit exponent $\frac{1}{4n-1}$.
Liu~\cite{Liu09} subsequently improved this exponent to $\frac{1}{2n-1}$,
which is optimal in the class of merely bounded densities: there exist bounded densities for which the potential is exactly of class $C^{1,\frac{1}{2n-1}}$ and
no better, so the exponent cannot be raised without further hypotheses (see~\cite{Lo09,Liu09}).
Under the additional assumption that the densities are H\"older continuous or smooth, the (MTW)
theory upgrades the potential to $C^{2,\alpha}$ and to $C^\infty$ respectively, by the work of
Ma--Trudinger--Wang~\cite{MTW05}, Loeper~\cite{Lo09}, and Trudinger--Wang~\cite{TW09}.

Theorem~\ref{thm:A} reaches the same critical exponent $\frac{1}{2n-1}$ from the opposite
direction. In our example the right-hand side of~\eqref{eq:ot-pde} is not merely bounded but
\emph{constant}; were the relativistic cost to satisfy
(MTW), the Loeper--Liu theory just quoted would force the potential to be $C^\infty$. The loss
of regularity is therefore attributable entirely to the geometry of the cost: the relativistic
cost~\eqref{eq:cost} \emph{violates} (MTW) on its whole finite domain
$|x-y|<a$; see the calculation in~\cite[Section~6]{QZ26}. This failure permits a
generalized solution that, despite smooth data, is no smoother than the worst potentials
allowed by Liu's sharp exponent for rough data. Thus the exponent $\frac{1}{2n-1}$ turns out to
be critical in two independent regimes: rough data with an (MTW) cost, and smooth data with
the relativistic cost.
Theorem~\ref{thm:A} shows that, for the relativistic cost, it
cannot be improved by any amount of regularity of the right-hand side.

Our second result applies this construction to a geometric equation. Fix $\Theta\in(0,\pi/2)$
and consider the two-dimensional special Lagrangian curvature equation with phase $\Theta$,
\begin{equation}\label{eq:slce}
    \arctan\kappa_1+\arctan\kappa_2=\Theta,
\end{equation}
where $\kappa_1,\kappa_2$ are the principal curvatures of a surface $\mathcal{M}\subset\R^3$.
Introduced by Smith~\cite{S13} as the curvature analogue of the Harvey--Lawson special
Lagrangian equation~\cite{HL82}, it admits non-convex graphical solutions, and its interior regularity is
delicate. Related interior curvature estimates for three-dimensional hypersurfaces of prescribed
scalar curvature were established by Qiu~\cite{Qiu24}. For the special Lagrangian curvature
equation, Qiu and Zhou~\cite{QZ26} proved a priori interior estimates under structural
hypotheses, and their Lemma~6.1 (our Lemma~\ref{lem:ot-slce}) links~\eqref{eq:slce}, for a
suitable choice of parameters, to the relativistic equation~\eqref{eq:ot-pde}. Exploiting this
link together with Theorem~\ref{thm:A}, we realize the singular profile as a limit of smooth
solutions, in the spirit of the singular special Lagrangian solutions of Wang--Yuan~\cite{WY13}.

\begin{theorem}\label{thm:B}
    On a ball $B\subset\R^2$ there exist smooth graphical solutions $u_j$ $(j\geq1)$ of
    \eqref{eq:slce} converging in $C^0(B)$ to a limit $u_\infty$ that belongs to $C^{1,1/3}(B)$
    but not to $C^{1,\beta}(B)$ for any $\beta>1/3$. In particular, \eqref{eq:slce} admits no
    pure interior curvature estimate.
\end{theorem}

Here {\it no pure interior curvature estimate} means that $D^2u(0)$ cannot be
controlled by $\|u\|_{C^0(B)}$ alone: the $u_j$ stay uniformly bounded
while their gradients lose $\beta$-H\"older regularity at the center. The exponent $1/3$ is the case
$n=2$ of Theorem~\ref{thm:A}, here attained as a uniform limit of {\it smooth} solutions.

This should be contrasted with the two-dimensional Monge--Amp\`ere equation. For convex
solutions of $\det D^2u=f>0$ in the plane, pure interior second-derivative estimates are
available: the classical estimate is due to Heinz~\cite{He59}, and more recent proofs were
given by Chen--Han--Ou~\cite{CHO16} and Liu~\cite{Liu20}. In those results the
interior Hessian is controlled in terms of the $C^0$ size of the solution and the prescribed
data, with no separate gradient bound. Theorem~\ref{thm:B} shows that this feature does not
persist for the special Lagrangian curvature equation: even for smooth graphical solutions
with fixed phase, uniform $C^0$ control does not prevent the loss of gradient H\"older
regularity.

The proof of Theorem~\ref{thm:A} proceeds in three steps.
First a formal calculation shows that the radial ansatz
$u(x',t)=\sqrt{r(t)^2-|x'|^2}$ reduces~\eqref{eq:ot-pde} to the scalar second order
ODE~\eqref{eq:ode} for the meridian $r$.
Imposing $r(0)=a$, $\dot r(0)=0$ forces the radial factor $\mathcal{A}=1-a/(r\sqrt{1+\dot r^2})$ to
vanish at $t=0$, which in turn forces $\ddot r$ to blow up; tracking the leading asymptotics
$\mathcal{A}\sim s^2/2$ with $s=\dot r$ yields $\dot s\sim s^{-2(n-1)}$ and hence $|s|\sim|t|^{1/(2n-1)}$,
the source of the exponent.
Second, to make this rigorous through the degeneracy at $t=0$ we recast the dynamics as a
first order autonomous system in the phase variable $s$, equations~\eqref{eq:ode-sys}, whose
right-hand side is smooth at $s=0$; the Picard--Lindel\"of theorem then produces a smooth solution, and the inversion $t\mapsto s(t)$ supplies the sharp asymptotics
in~\eqref{eq:s-asymp}.
Third, we verify that the resulting $u$ is an Aleksandrov solution (Lemma~\ref{lem:gs}): $c$-convexity is
obtained from a one-dimensional support inequality for $r$ promoted to the full graph by a reverse triangle inequality for the relativistic cost (Lemma~\ref{lem:in-CS}), and the
Monge--Amp\`ere measure is computed exactly by slicing the image of $T_u$.
Finally, the sharp $C^{1,\frac{1}{2n-1}}$ regularity of $u$ is checked in Lemma~\ref{lem:reg}.

Theorem~\ref{thm:B} is then obtained by a perturbation argument:
the autonomous system is perturbed at the initial value, as in~\eqref{eq:sys-ve}, so that each perturbed solution is everywhere {\it smooth} and, via Lemma~\ref{lem:ot-slce}, gives a
smooth graphical solution of~\eqref{eq:slce} for the parameters $a=\tan\Theta$,
$\lambda=1/\cos^2\Theta$; continuous dependence on initial data then forces the perturbed
solutions to converge in $C^0$ to the singular profile of Theorem~\ref{thm:A} with $n=2$.

Section~\ref{sec:prelim} fixes the optimal-transport vocabulary ($c$-linear and $c$-convex
functions, the $c$-normal map, generalized solutions) and proves the reverse triangle
inequality for the relativistic cost.
Section~\ref{sec:thmA} carries out the construction and proves Theorem~\ref{thm:A}.
Section~\ref{sec:thmB} performs the perturbation and proves Theorem~\ref{thm:B}.

\subsection*{Notation}
Throughout, $x=(x',t)\in\R^{n-1}\times\R$, $B_\eta(0)$ denotes the open Euclidean ball of
radius $\eta$ centered at the origin, and $|\cdot|$ denotes both Euclidean norm and Lebesgue
measure, the meaning being clear from context.
A dot denotes differentiation with respect to $t$ and a prime differentiation with respect to
the phase variable $s$.
We reserve $a>0$ for the relativistic cutoff and $\lambda>0$
for the right-hand side of~\eqref{eq:ot-pde}.

\section{Preliminaries}\label{sec:prelim}

This section fixes the notions from the theory of optimal transportation that we use,
specialized to the relativistic cost~\eqref{eq:cost}, and proves the one elementary inequality used in the verification of $c$-convexity in Section~\ref{sec:thmA}. For background on
$c$-convexity, optimal maps, and the Aleksandrov formulation of the associated
Monge--Amp\`ere equation, see~\cite{Vi09,Fi17}.
Throughout, $\Om\subset\R^n$ is a domain and $c$ is the cost~\eqref{eq:cost}.

\subsection{\texorpdfstring{$c$}{c}-convexity and the \texorpdfstring{$c$}{c}-normal map}

The role played by affine functions in classical convexity is played here by the translates of
the cost.

\begin{definition}\label{def:c-linear}
    A function $\ell$ is {\bf $c$-linear} if $\ell(\cdot)=c(\cdot,y)+b$ for some $y\in\R^n$ and
    $b\in\R$.
\end{definition}

A $c$-linear function is thus a vertical translate of a single sheet $c(\cdot,y)$ of the cost,
which for the relativistic cost~\eqref{eq:cost} is a lower hemisphere of radius $a$ centered at
$(y,0)$. Replacing supporting hyperplanes by supporting $c$-linear functions yields the notion
of $c$-convexity.

\begin{definition}\label{def:c-convex}
    A lower semi-continuous function $\phi$ on $\Om$ is {\bf $c$-convex} in $\Om$ if for every
    $x_0\in\Om$ there is a $c$-linear function $\ell$ that supports $\phi$ from below at $x_0$,
    that is, $\ell\leq\phi$ in $\Om$ and $\ell(x_0)=\phi(x_0)$. Such an $\ell$ is called a
    {\bf $c$-support} of $\phi$ at $x_0$.
\end{definition}

The collection of points $y$ that generate supporting $c$-linear functions at a given $x_0$
is the analogue of the subdifferential of a convex function; recording it defines the
$c$-normal map, which plays the role of the gradient map.

\begin{definition}[$c$-normal map]\label{def:c-normal}
    Let $\phi$ be $c$-convex on $\Om$. Its {\bf $c$-normal map} is the set-valued map
    $T_\phi:\Om\rightrightarrows\R^n$ assigning to $x_0\in\Om$ the set of all $y_0\in\R^n$ for
    which $c(\cdot,y_0)+b$ is a $c$-support of $\phi$ at $x_0$ for some $b\in\R$.
\end{definition}

\begin{remark}\label{rk:c-normal}
    If a $c$-convex function $\phi$ is $C^1$, then $T_\phi$ is single-valued and is determined
    by the first-order matching condition $\D_xc(x,T_\phi(x))=\D\phi(x)$. For the relativistic
    cost~\eqref{eq:cost} this is exactly the explicit formula~\eqref{eq:T_u}.
\end{remark}

\subsection{Generalized solutions}

Pushing Lebesgue measure forward by the $c$-normal map turns $c$-convex functions into measures,
which lets us make sense of~\eqref{eq:ot-pde} for non-smooth $\phi$.

\begin{definition}\label{def:gs}
    A $c$-convex function $u$ on $\Om$ induces a {\bf Monge--Amp\`ere measure} $\mu_u$ by
    $\mu_u(E):=|T_u(E)|$ for Borel sets $E\subset\Om$, where $|\cdot|$ is the Lebesgue measure
    on $\R^n$. We say that $u$ is a {\bf generalized solution} (or {\bf Aleksandrov solution}) of
    \eqref{eq:ot-pde} if $\mu_u=\lambda\,dx$, that is, if
    \begin{equation}\label{eq:T_u-id}
        |T_u(E)|=\lambda|E|\qquad\text{for every Borel set }E\subset\Om.
    \end{equation}
\end{definition}

\subsection{A reverse triangle inequality for the relativistic cost}

The verification that our profile is $c$-convex reduces a multidimensional support inequality to
a one-dimensional one by means of the following superadditivity property of the map
$(p,v)\mapsto\sqrt{p^2-|v|^2}$ on the forward cone $\{p\geq|v|\}$. Geometrically it is the
reverse triangle inequality of Lorentzian geometry, and it is exactly what makes the
relativistic cost compatible with the radial construction.

\begin{lemma}\label{lem:in-CS}
    Let $p,q>0$ and $v,w\in\R^{n-1}$ satisfy $p\geq|v|$ and $q\geq|w|$. Then
    \begin{equation}\label{eq:in-CS}
        \sqrt{(p+q)^2-|v+w|^2}\geq\sqrt{p^2-|v|^2}+\sqrt{q^2-|w|^2}.
    \end{equation}
\end{lemma}
\begin{proof}
    Both sides of~\eqref{eq:in-CS} are non-negative, so squaring shows that~\eqref{eq:in-CS} is
    equivalent to
    \[
        pq-v\cdot w\geq\sqrt{p^2-|v|^2}\,\sqrt{q^2-|w|^2}.
    \]
    Since $v\cdot w\leq|v|\,|w|$ by the Cauchy--Schwarz inequality, it suffices to prove the
    stronger inequality
    \[
        pq-|v|\,|w|\geq\sqrt{p^2-|v|^2}\,\sqrt{q^2-|w|^2}.
    \]
    Here the left-hand side is non-negative because $pq\geq|v|\,|w|$ by hypothesis, so we may
    square once more; the resulting inequality simplifies to $(p\,|w|-q\,|v|)^2\geq0$, which
    holds. This proves~\eqref{eq:in-CS}.
\end{proof}

\section{Proof of Theorem~\ref{thm:A}}\label{sec:thmA}

\subsection{Reduction to an ordinary differential equation}\label{ss:reduction}

We seek a solution of \eqref{eq:ot-pde} among the rotationally symmetric profiles
\begin{equation}\label{eq:u}
    u(x',t)=\sqrt{r(t)^2-|x'|^2},\qquad x=(x',t)\in\R^{n-1}\times\R,
\end{equation}
where $r$ is a positive function of $t$ alone and, as before, a dot denotes $d/dt$.
The graph of such a $u$ is a hypersurface of revolution about the $t$-axis with meridian $r$.

From \eqref{eq:u} we have $\D u=(-x'/u,r\dot r/u)$, and substituting into the
explicit formula \eqref{eq:T_u} for the $c$-normal map gives
\begin{equation}\label{s3:e1}
    T_u(x',t)=(A(t)x',B(t)),
\end{equation}
where
\begin{equation}\label{eq:A-B}
    A(t)=1-\frac{a}{r\sqrt{1+\dot{r}^2}},\qquad
    B(t)=t+\frac{a\dot{r}}{\sqrt{1+\dot{r}^2}}.
\end{equation}
Differentiating \eqref{s3:e1} produces the block upper-triangular Jacobian
\begin{equation}\label{eq:DT_u}
    DT_u=
    \begin{pmatrix}
        A(t)I_{n-1} & \dot{A}(t)x'\\
        0 & \dot{B}(t)
    \end{pmatrix},
    \qquad
    \det DT_u=A(t)^{n-1}\,\dot{B}(t).
\end{equation}
Since $\dot B=1+a\ddot r/(1+\dot r^2)^{3/2}$ and \eqref{eq:ot-pde} requires $\det DT_u=\lambda$,
the profile \eqref{eq:u} solves \eqref{eq:ot-pde} if and only if the meridian $r$ satisfies the
second order ordinary differential equation
\begin{equation}\label{eq:ode}
    \left(1-\frac{a}{r\sqrt{1+\dot{r}^2}}\right)^{n-1}
    \left(1+\frac{a\ddot{r}}{(1+\dot{r}^2)^{3/2}}\right)=\lambda.
\end{equation}
We single out the initial conditions
\begin{equation}\label{eq:ic}
    r(0)=a,\qquad \dot r(0)=0,
\end{equation}
for which the radial factor satisfies $A(0)=0$. This degeneracy is the seed of the singular
behaviour: with $A(0)=0$, equation \eqref{eq:ode} can hold only if $\ddot r$ blows up at $t=0$.
Turning the relation \eqref{eq:ode}--\eqref{eq:ic} into a genuine solution, and establishing its
regularity, is the content of Sections~\ref{ss:system}--\ref{ss:verify}.

\begin{remark}[Heuristics for the H\"older exponent]\label{rk:heuristic}
    The value $\tfrac1{2n-1}$ in Theorem~\ref{thm:A} can be read off from
    \eqref{eq:ode}--\eqref{eq:ic} by a scaling heuristic, which we record for orientation only
    and do not use in the rigorous argument. Set $s=\dot r$. Near $t=0$ one expects
    $A\sim s^2/2$ as $s\to0$ (this is proved in Lemma~\ref{lem:A}). Since $\ddot r=\dot s$,
    equation \eqref{eq:ode} then reads, to leading order,
    $(s^2/2)^{n-1}(1+a\dot s)\approx\lambda$, whence $\dot s\sim s^{-2(n-1)}$. Integrating in
    $t$ gives $|t|\sim|s|^{2n-1}$, that is $|s|\sim|t|^{1/(2n-1)}$; as $\partial_t u(0,t)=s(t)$,
    this is the gradient H\"older exponent $\tfrac1{2n-1}$.
\end{remark}

\subsection{The autonomous system and its asymptotics}\label{ss:system}

In this subsection, we consider the following ODE system transformed from \eqref{eq:ode}:
\begin{equation}\label{eq:ode-sys}
    \left\{
        \begin{aligned}
            \frac{dt}{ds}&=\frac{a\mathcal{A}^{n-1}}{(\lambda-\mathcal{A}^{n-1})(1+s^2)^{3/2}},\\
            \frac{dR}{ds}&=s\frac{a\mathcal{A}^{n-1}}{(\lambda-\mathcal{A}^{n-1})(1+s^2)^{3/2}},\\
            t(0)&=0,R(0)=a.
        \end{aligned}
    \right.
\end{equation}
Here the function $\mathcal{A}$ in $s$ is given by
\begin{equation}\label{eq:A(s)}
    \mathcal{A}(s):=1-\frac{a}{R(s)\sqrt{1+s^2}}.
\end{equation}
The initial condition in \eqref{eq:ode-sys} gives $\mathcal{A}(0)=0$.
Since $\lambda>0$, the right-hand side of \eqref{eq:ode-sys} is smooth near $s=0$.
Then there exists a sufficiently small $\delta>0$ so that \eqref{eq:ode-sys} admits a unique $C^\infty$ solution $s\mapsto(t(s),R(s))$ on $(-\delta,\delta)$.
In particular, the function $\mathcal{A}$ is smooth on $(-\delta,\delta)$.
For simplicity of notation, we will use $'$ to denote the differentiation in $s$.

\begin{lemma}\label{lem:A}
    $\mathcal{A}(s)=s^2/2+o(s^2)$.
\end{lemma}
\begin{proof}
    We already know that $\mathcal{A}(0)=0$.
    Differentiating \eqref{eq:A(s)}, we have
    \[
    \mathcal{A}'(s)=a\left(\frac{R'}{R^2\sqrt{1+s^2}}+\frac{s}{R(1+s^2)^{3/2}}\right)\quad\implies\quad\mathcal{A}'(0)=0.
    \]
    The first and the second equations in \eqref{eq:ode-sys} imply $R'=st'$.
    Hence
    \[
    R''=st''+t'\quad\implies\quad R''(0)=0.
    \]
    A direct calculation gives $\mathcal{A}''(0)=1$.
    Lemma~\ref{lem:A} follows immediately.
\end{proof}

\begin{lemma}\label{lem:inverse}
    After possibly decreasing $\delta>0$, the function $s\mapsto t(s)$ is strictly increasing on $(-\delta,\delta)$.
    Then there exists a unique inverse function $t\mapsto s(t)$ defined on some interval $(-T,T)$.
    Moreover, the function $s$ satisfies the following asymptotic formula near $t=0$:
    \begin{equation}\label{eq:s-asymp}
        s(t)=\gamma\operatorname{sgn}(t)|t|^{\frac{1}{2n-1}}+o(|t|^{\frac{1}{2n-1}}),
    \end{equation}
    where $\gamma$ is a positive constant depending only on $n,a,\lambda$.
\end{lemma}
\begin{proof}
    Plugging $\mathcal{A}(s)=s^2/2+o(s^2)$ (see Lemma~\ref{lem:A}) into the first equation in \eqref{eq:ode-sys},
    \begin{equation}\label{eq:t'-asymp}
        t'(s)=a\lambda^{-1}2^{-(n-1)}s^{2(n-1)}+o(s^{2(n-1)}).
    \end{equation}
    After possibly decreasing $\delta>0$, one has $t'(s)>0$ on $(-\delta,\delta)\setminus\{0\}$.
    Using $\mathcal{A}(0)=0$, we have $t'(0)=0$.
    Hence the function $s\mapsto t(s)$ is strictly increasing on $(-\delta,\delta)$.
    Then there exists an inverse function $t\mapsto s(t)$ defined on $(-T,T)$, where $T:=\min\{-t(-\delta),t(\delta)\}>0$.

    Integrating \eqref{eq:t'-asymp}, we have the asymptotic formula for the function $t$:
    \begin{equation}\label{eq:t-asymp}
        t(s)=a\lambda^{-1}2^{-(n-1)}(2n-1)^{-1}s^{2n-1}+o(s^{2n-1}).
    \end{equation}
    Inverting this leading-order relation gives \eqref{eq:s-asymp} with
    \[
    \gamma:=(a^{-1}\lambda 2^{n-1}(2n-1))^{\frac{1}{2n-1}}.
    \]
\end{proof}

\subsection{The singular profile and the choice of radius}\label{ss:profile}

By Lemma~\ref{lem:inverse}, the function $s(t)$ is defined and continuous on $(-T,T)$ with
$s(0)=0$. Set
\[
r(t):=R(s(t)),\qquad t\in(-T,T),
\]
so that $r$ is continuous and $r(0)=R(0)=a$.

\begin{lemma}\label{lem:C1}
    The function $r$ is $C^1$ on $(-T,T)$ and $\dot{r}=s$ there.
\end{lemma}
\begin{proof}
   If $t\in(-T,T)\setminus\{0\}$, then $\dot{s}(t)\neq 0$. The implicit function theorem yields that the function $t\mapsto s(t)$ is smooth and satisfies $\dot{s}=1/t'$.
   The first and the second equations in \eqref{eq:ode-sys} give $R'=st'$.
  Using the chain rule,
   \[
    \dot{r}=R'\dot{s}=st'(t')^{-1}=s.
   \]
   If $t=0$, we verify that $\dot{r}(0)=0$ by definition. In fact,
   \[
   \lim_{t\to 0}\frac{r(t)-r(0)}{t}=\lim_{s\to 0}\frac{R(s)-R(0)}{t(s)}=\lim_{s\to 0}\frac{R'}{t'}=0,
   \]
   where the fact $R'=st'$ is used in the last equality and $s(0)=0$.
   We conclude that $r$ is $C^1$ on $(-T,T)$ and $\dot{r}=s$ there.
\end{proof}

We next record two monotonicity facts that hold on the whole interval $(-T,T)$ and require no
restriction of its length.

\begin{lemma}\label{lem:Ra}
    For every $t\in(-T,T)$ one has $r(t)\geq a$ and $A(t)\in[0,1)$, where $A$ is the radial
    factor in \eqref{eq:A-B}.
\end{lemma}
\begin{proof}
    By Lemma~\ref{lem:inverse} we have $t'(s)\geq0$, and the first two equations
    of~\eqref{eq:ode-sys} give $R'(s)=s\,t'(s)$, which therefore has the same sign as $s$.
    Hence $R$ is non-increasing for $s<0$ and non-decreasing for $s>0$, so $R$ attains its
    minimum at $s=0$ and $R(s)\geq R(0)=a$ for all $s$. In particular $r(t)=R(s(t))\geq a$.
    Consequently $r\sqrt{1+s^2}\geq a$, and by Lemma~\ref{lem:C1} (so that $\dot r=s$),
    \[
        A(t)=1-\frac{a}{r(t)\sqrt{1+s(t)^2}}\in[0,1).\qedhere
    \]
\end{proof}

We now fix, once and for all, the radius of the ball on which the singular profile is
constructed; it will \emph{not} be decreased anywhere in the remainder of the proof.

\begin{lemma}[Choice of radius]\label{lem:radius}
    There exists $\eta\in(0,\min\{T,a/8\})$ such that
    \begin{equation}\label{eq:R-cond}
        \frac{a\,|s(t)|}{\sqrt{1+s(t)^2}}\leq\frac a4\qquad\text{for all }t\in[-\eta,\eta].
    \end{equation}
    Fix such an $\eta$ and write $B_\eta:=B_\eta(0)\subset\R^n$. Then fixing an arbitrary point $(x'_0,t_0)\in B_\eta$, with the notations
    $r_0=r(t_0)$, $A_0=A(t_0)$ and $B_0=B(t_0)$:
    \begin{enumerate}
        \item[\rm(i)] the function $u(x',t):=\sqrt{r(t)^2-|x'|^2}$ is well-defined and positive
        on $B_\eta$;
        \item[\rm(ii)] $|t-B_0|\leq a/2$, and hence $a^2-(t-B_0)^2\geq 3a^2/4$, for all
        $t,t_0\in[-\eta,\eta]$;
        \item[\rm(iii)] $\sqrt{a^2-(t-B_0)^2}\geq|x'-A_0x_0'|$ for all
        $(x',t),(x_0',t_0)\in B_\eta$.
    \end{enumerate}
\end{lemma}
\begin{proof}
    Since $s$ is continuous with $s(0)=0$, the left-hand side of \eqref{eq:R-cond} tends to $0$
    as $t\to0$.
Hence there is $\eta\in(0,\min\{T,a/8\})$ satisfying \eqref{eq:R-cond} on
    $[-\eta,\eta]$. Fix such an $\eta$.

    (i) By Lemma~\ref{lem:Ra}, $r\geq a$, so for $(x',t)\in B_\eta$ we have
    $r(t)^2-|x'|^2\geq a^2-a^2/64>0$.

    (ii) Recall $B_0=t_0+a\,s(t_0)/\sqrt{1+s(t_0)^2}$, so \eqref{eq:R-cond} gives
    $|t_0-B_0|\leq a/4$. Therefore, for $t,t_0\in[-\eta,\eta]$,
    \[
        |t-B_0|\leq|t-t_0|+|t_0-B_0|\leq 2\eta+\frac a4\leq\frac a4+\frac a4=\frac a2,
    \]
    using $\eta\leq a/8$. Hence $a^2-(t-B_0)^2\geq a^2-a^2/4=3a^2/4$.

    (iii) By Lemma~\ref{lem:Ra}, $A_0\in[0,1)$, so for $(x',t),(x_0',t_0)\in B_\eta$,
    \[
        |x'-A_0x_0'|\leq|x'|+A_0|x_0'|<2\eta\leq\frac a4\leq\frac{\sqrt3}{2}a
        \leq\sqrt{a^2-(t-B_0)^2},
    \]
    where the final inequality is (ii).
\end{proof}

Throughout the rest of this section, $\eta>0$ is fixed as in Lemma~\ref{lem:radius} and
\begin{equation}\label{eq:u-sol}
    u(x',t):=\sqrt{r(t)^2-|x'|^2},\qquad(x',t)\in B_\eta,
\end{equation}
denotes the corresponding profile.
By Lemma~\ref{lem:C1} and \eqref{eq:u-sol}, $u$ is $C^1$ smooth on $B_\eta$, and Remark~\ref{rk:c-normal} implies that the $c$-normal map $T_u$ of $u$ is single-valued and equals \eqref{s3:e1}.

\subsection{Verification of the Aleksandrov solution and regularity}\label{ss:verify}

\begin{lemma}\label{lem:gs}
    The function $u$ is a generalized solution to \eqref{eq:ot-pde}.
\end{lemma}
\begin{proof}
    {\it Step 1: $u$ is $c$-convex in $B_\eta$.}

    It suffices to show that for an arbitrary point $(x_0',t_0)\in B_\eta$,
    there exists a $c$-support of $u$ at $(x_0',t_0)$.
    Note that $T_u(x',t)=(A(t)x',B(t))$, where $A$ and $B$ are given in \eqref{eq:A-B}.
    For simplicity of notation, set $r_0=r(t_0)$, $A_0=A(t_0)$ and $B_0=B(t_0)$.
    We will establish a one-dimensional support inequality for the function $r(t)$ and then obtain the support inequality for $u(x',t)$ by virtue of Lemma~\ref{lem:in-CS}.

    By Lemma~\ref{lem:radius}(ii), the function $g(t):=A_0r_0+\sqrt{a^2-(t-B_0)^2}$ is well-defined on $(-\eta,\eta)$.
    Define $h(t):=r(t)-g(t)$.
    Then $h(t_0)=\dot{h}(t_0)=0$.
    We see that $\dot{g}$ is strictly decreasing because $\ddot{g}(t)=-a^2/(a^2-(t-B_0)^2)^{3/2}<0$.
    On the other hand, $\dot{r}=s$ is strictly increasing by Lemma~\ref{lem:inverse}.
    Putting these facts together, we further see that $h$ attains its minimum at $t=t_0$.
    It follows that for every $t\in(-\eta,\eta)$ there holds
    \begin{equation}\label{eq:r-support}
        r(t)\geq A_0r_0+\sqrt{a^2-(t-B_0)^2}.
    \end{equation}

    Since $A_0\geq0$ by Lemma~\ref{lem:Ra} and $r_0\geq|x_0'|$, we have $A_0r_0\geq A_0|x_0'|$.
    On the other hand, by Lemma~\ref{lem:radius}(iii), $\sqrt{a^2-(t-B_0)^2}\geq|x'-A_0x_0'|$ on $B_\eta$.
    Then
    \begin{align*}
       &u(x',t)-\Big(u(x_0',t_0)+c((x',t),T_u(x_0',t_0))-c((x_0',t_0),T_u(x_0',t_0))\Big)\\
       =&\sqrt{r(t)^2-|x'|^2}-A_0\sqrt{r_0^2-|x_0'|^2}-\sqrt{a^2-|x'-A_0x_0'|^2-(t-B_0)^2}\\
       \geq&\sqrt{\Big(A_0r_0+\sqrt{a^2-(t-B_0)^2}\Big)^2-|A_0x_0'+(x'-A_0x_0')|^2}\\
       &-\sqrt{(A_0r_0)^2-|A_0x_0'|^2}-\sqrt{a^2-(t-B_0)^2-|x'-A_0x_0'|^2}\geq0,
    \end{align*}
    where \eqref{eq:r-support} is used in the second-to-last inequality and Lemma~\ref{lem:in-CS} is used in the last inequality.
    Define
    \[
    \ell(x',t):=u(x_0',t_0)+c((x',t),T_u(x_0',t_0))-c((x_0',t_0),T_u(x_0',t_0)).
    \]
    Then $\ell(x_0',t_0)=u(x_0',t_0)$ and $\ell\leq u$ on $B_\eta$, namely $\ell$ is a $c$-support of $u$ at $(x_0',t_0)$.
    Since $(x_0',t_0)$ is arbitrary, we conclude that $u$ is $c$-convex in $B_\eta$.

    {\it Step 2: $u$ satisfies \eqref{eq:T_u-id} for every Borel set $E\subset B_\eta$.}

    Fix an arbitrary Borel set $E\subset B_\eta$.
    Recall that $T_u$ is a single-valued map given by \eqref{s3:e1}.
    We reparametrize $T_u$ via $s$.
    For simplicity of notation, set
    \begin{align*}
        &\mathcal{A}(s):=A(t(s))=1-\frac{a}{R(s)\sqrt{1+s^2}},\\
    &\mathcal{B}(s):=B(t(s))=t(s)+\frac{as}{\sqrt{1+s^2}}.
    \end{align*}
    Then $T_u(x',t(s))=(\mathcal{A}(s)x',\mathcal{B}(s))$.
    Since $t'\geq 0$ and $a>0$, we have
    \[
    \mathcal{B}'(s)=t'+\frac{a}{(1+s^2)^{3/2}}>0.
    \]
    Hence $\mathcal{B}$ is a strictly increasing smooth function on $(s(-\eta),s(\eta))$.
    Let $\mathcal{B}^{-1}$ denote the inverse function of $\mathcal{B}$.

    We calculate $|T_u(E)|$ by slicing.
    Given a height $w\in(\mathcal{B}(s(-\eta)),\mathcal{B}(s(\eta)))$, the $w$-slice of $T_u(E)$ is given by
    \[
    (T_u(E))_w:=\{x'\in\R^{n-1}:(x',w)\in T_u(E)\}=\mathcal{A}(\mathcal{B}^{-1}(w))E_{t(\mathcal{B}^{-1}(w))},
    \]
    where $E_{t(\mathcal{B}^{-1}(w))}:=\{x'\in\R^{n-1}:(x',t(\mathcal{B}^{-1}(w)))\in E\}$.
    It follows that
    \begin{equation}\label{eq:slice}
        |(T_u(E))_w|=(\mathcal{A}(\mathcal{B}^{-1}(w)))^{n-1}|E_{t(\mathcal{B}^{-1}(w))}|.
    \end{equation}
    Using the first and the second equations in \eqref{eq:ode-sys}, we obtain
   \begin{equation}\label{eq:ode-id}
    (\mathcal{A}(s))^{n-1}\mathcal{B}'(s)=\lambda t'(s),\quad\forall s\in(-s(\eta),s(\eta)).
   \end{equation}
   Using \eqref{eq:slice} and \eqref{eq:ode-id} in order, we have
   \begin{align*}
    |T_u(E)|=&\int_{\mathcal{B}(s(-\eta))}^{\mathcal{B}(s(\eta))}|(T_u(E))_w|dw\\
    =&\int_{\mathcal{B}(s(-\eta))}^{\mathcal{B}(s(\eta))}(\mathcal{A}(\mathcal{B}^{-1}(w)))^{n-1}|E_{t(\mathcal{B}^{-1}(w))}|dw\\
    =&\int_{s(-\eta)}^{s(\eta)}(\mathcal{A}(s))^{n-1}|E_{t(s)}|\mathcal{B}'(s)ds\\
    =&\lambda\int_{s(-\eta)}^{s(\eta)}|E_{t(s)}|t'(s)ds=\lambda|E|.
   \end{align*}
   Since $E$ is arbitrary, we conclude that \eqref{eq:T_u-id} holds for every Borel set in $B_\eta$.

   Putting Step 1 and Step 2 together, Lemma~\ref{lem:gs} is proved.
\end{proof}

\begin{lemma}\label{lem:reg}
    $u\in C^{1,\frac{1}{2n-1}}(B_\eta)$ but $u\notin C^{1,\beta}(B_\eta)$ whenever $\beta>\frac{1}{2n-1}$.
\end{lemma}
\begin{proof}
    Set $\alpha:=\frac1{2n-1}$.

    {\it Step 1: $u\in C^{1,\alpha}(B_\eta)$.} By Lemma~\ref{lem:C1}, $u\in C^1(B_\eta)$ with
    \[
        \D u=\Big(-\frac{x'}{u},\ \frac{r\,s}{u}\Big),\qquad s=\dot r .
    \]
    By Lemma~\ref{lem:Ra}, $r\geq a$, and since $|x'|<\eta<a$ we have
    $u\geq\sqrt{a^2-\eta^2}=:c_0>0$ on $B_\eta$. Hence
    $\D u(x',t)=\Phi\big(x',r(t),s(t)\big)$, where
    \[
\Phi(x',\rho,\sigma)=\big(-x'/\sqrt{\rho^2-|x'|^2},\ \rho\sigma/\sqrt{\rho^2-|x'|^2}\big)
    \]
    is
    smooth, with bounded first derivatives, on the compact set
    \[
\{|x'|\leq\eta,\ a\leq\rho\leq\max_{[-\eta,\eta]}r,\ |\sigma|\leq\max_{[-\eta,\eta]}|s|\},
    \]
    where $\sqrt{\rho^2-|x'|^2}\geq c_0$. Thus $\Phi$ is Lipschitz, so $\D u$ inherits the
    modulus of continuity of $t\mapsto(r(t),s(t))$. As $r\in C^1$ is Lipschitz, it remains to show
    \begin{equation}\label{eq:s-holder}
        s\in C^{0,\alpha}([-\eta,\eta]).
    \end{equation}
    Lemma~\ref{lem:C1} yields that the function $s$ is $C^1$ on $[-\eta,\eta]\setminus\{0\}$ with $\dot s=1/t'$.
    On the other hand, Lemma~\ref{lem:inverse} implies that the function $s$ is $\alpha$-H\"older continuous at $t=0$.
    Hence \eqref{eq:s-holder} is true and we conclude that $u\in C^{1,\alpha}(B_\eta)$.

    {\it Step 2: $u\notin C^{1,\beta}(B_\eta)$ for $\beta>\alpha$.} Along the axis $x'=0$ we have
    $u(0,t)=r(t)$, so $\partial_t u(0,t)=s(t)$ and $\partial_t u(0,0)=s(0)=0$. By
    \eqref{eq:s-asymp},
    \[
        \frac{|\partial_t u(0,t)-\partial_t u(0,0)|}{|t|^{\beta}}
        =\frac{|s(t)|}{|t|^{\beta}}=\gamma\,|t|^{\alpha-\beta}+o(|t|^{\alpha-\beta})
        \xrightarrow[\;t\to0\;]{}\infty .
    \]
    Hence $\partial_t u$ is not $\beta$-H\"older continuous at the origin, so
    $u\notin C^{1,\beta}(B_\eta)$.
\end{proof}

Putting Lemma~\ref{lem:gs} and Lemma~\ref{lem:reg} together, we complete the proof of Theorem~\ref{thm:A}.

\section{Proof of Theorem~\ref{thm:B}}\label{sec:thmB}

The bridge between the relativistic optimal transport equation \eqref{eq:ot-pde} and the
special Lagrangian curvature equation \eqref{eq:slce} is the following lemma of Qiu and Zhou.

\begin{lemma}[{\cite[Lemma~6.1]{QZ26}}]\label{lem:ot-slce}
    Fix $\Theta\in(0,\pi/2)$. If $u$ is a smooth solution of \eqref{eq:ot-pde} with
    $a=\tan\Theta$ and $\lambda=1/\cos^2\Theta$ on a ball $B\subset\R^2$, then the graph
    $\mathcal{M}:=\{(z,u(z)):z\in B\}$ is a smooth solution of \eqref{eq:slce}.
\end{lemma}

By Lemma~\ref{lem:ot-slce}, Theorem~\ref{thm:B} reduces to producing \emph{smooth} solutions of
\eqref{eq:ot-pde} in dimension $n=2$ that converge uniformly to the singular profile of
Theorem~\ref{thm:A}. We obtain them by perturbing the initial value in \eqref{eq:ode-sys} so as
to remove the degeneracy $A(0)=0$ that produced the singularity: raising $R(0)$ from $a$ to
$a+\varepsilon$ keeps the radial factor strictly positive, so the perturbed profile is smooth,
and letting $\varepsilon\to0$ restores the singular solution.

Throughout this section we fix $\Theta\in(0,\pi/2)$ and set $a=\tan\Theta$,
$\lambda=1/\cos^2\Theta$, and $n=2$. For $\varepsilon>0$, consider the system obtained from
\eqref{eq:ode-sys} by raising the initial value of $R$ by $\varepsilon$:
 \begin{equation}\label{eq:sys-ve}
    \left\{
        \begin{aligned}
            \frac{dt_\ve}{ds}&=\frac{a\mathcal{A}_\ve^{n-1}}{(\lambda-\mathcal{A}_\ve^{n-1})(1+s^2)^{3/2}},\\
            \frac{dR_\ve}{ds}&=s\frac{a\mathcal{A}_\ve^{n-1}}{(\lambda-\mathcal{A}_\ve^{n-1})(1+s^2)^{3/2}},\\
            t_\ve(0)&=0,R_\ve(0)=a+\ve.
        \end{aligned}
    \right.
\end{equation}
Here, in analogy with \eqref{eq:A(s)}, the function $\mathcal{A}_\ve$ is given by
\begin{equation}\label{eq:A-ve-s}
    \mathcal{A}_\ve(s):=1-\frac{a}{R_\ve(s)\sqrt{1+s^2}}.
\end{equation}
The displaced initial condition yields $\mathcal{A}_\ve(0)=\varepsilon/(a+\varepsilon)>0$, in
contrast with the unperturbed value $\mathcal{A}(0)=0$.
This strict positivity is the whole
point of the perturbation, as the next lemma makes precise.

\begin{lemma}\label{lem:perturbed}
    There exist $\varepsilon_0>0$, $\delta>0$ and $\eta>0$ such that for every
    $\varepsilon\in(0,\varepsilon_0)$ the following hold.
    \begin{enumerate}
        \item[\rm(i)] The system \eqref{eq:sys-ve} has a unique smooth solution
        $s\mapsto(t_\ve(s),R_\ve(s))$ on $(-\delta,\delta)$, and $\mathcal{A}_\ve(s)>0$ there.
        \item[\rm(ii)] The map $s\mapsto t_\ve(s)$ is strictly increasing with smooth inverse
        $t\mapsto s_\ve(t)$, and $r_\ve:=R_\ve\circ s_\ve$ satisfies $r_\ve\geq a$.
        \item[\rm(iii)] The function $u_\ve(x',t):=\sqrt{r_\ve(t)^2-|x'|^2}$ is a smooth solution
        of \eqref{eq:ot-pde} on the ball $B_\eta:=B_\eta(0)\subset\R^2$, and the graph
        $\mathcal{M}_\ve:=\{(z,u_\ve(z)):z\in B_\eta\}$ is a smooth solution of \eqref{eq:slce}.
    \end{enumerate}
\end{lemma}
\begin{proof}
    {\it (i).} Since $\mathcal{A}_\ve(0)=\varepsilon/(a+\varepsilon)>0$, the right-hand side of
    \eqref{eq:sys-ve} is smooth near $s=0$, so the Picard--Lindel\"of theorem yields a unique
    smooth solution on some interval $(-\delta,\delta)$. Shrinking $\varepsilon_0$ and $\delta$,
    we may assume $0<\mathcal{A}_\ve(s)<\tfrac12$, and hence $\mathcal{A}_\ve(s)^{n-1}<1<\lambda$,
    for all $\varepsilon\in(0,\varepsilon_0)$ and $|s|<\delta$, so the system stays
    non-degenerate. Differentiating \eqref{eq:A-ve-s} and using the second equation of
    \eqref{eq:sys-ve},
    \[
        \mathcal{A}_\ve'(s)
        =\frac{a}{R_\ve\sqrt{1+s^2}}\Big(\frac{R_\ve'}{R_\ve}+\frac{s}{1+s^2}\Big)
        =\frac{s\,(1-\mathcal{A}_\ve)(\lambda-\mathcal{A}_\ve^2)}{(1+s^2)(\lambda-\mathcal{A}_\ve)},
    \]
    so that $\operatorname{sgn}\mathcal{A}_\ve'(s)=\operatorname{sgn}(s)$. Together with
    $\mathcal{A}_\ve(0)>0$ this gives $\mathcal{A}_\ve>0$ on $(-\delta,\delta)$.

    {\it (ii).} By (i) and the first equation of \eqref{eq:sys-ve}, $t_\ve'(s)>0$ for all
    $|s|<\delta$, so $s\mapsto t_\ve(s)$ is strictly increasing and has a smooth inverse
    $t\mapsto s_\ve(t)$. The second equation gives $R_\ve'(s)=s\,t_\ve'(s)$, which has the sign
    of $s$; hence $R_\ve$ attains its minimum at $s=0$ and $R_\ve\geq R_\ve(0)=a+\varepsilon\geq a$,
    so $r_\ve=R_\ve\circ s_\ve\geq a$.

    {\it (iii).} By (i) and (ii), $R_\ve$ and $s_\ve$ are smooth, hence so is
    $r_\ve=R_\ve\circ s_\ve\geq a$. By continuous dependence on the initial data,
    $t_\ve(\pm\delta)\to t_0(\pm\delta)\neq0$ as $\varepsilon\to0$, where $t_0$ is the
    unperturbed solution of \eqref{eq:ode-sys}; hence, after shrinking $\varepsilon_0$, we may fix
    a single $\eta\in(0,a)$ with $[-\eta,\eta]\subset(t_\ve(-\delta),t_\ve(\delta))$ for every
    $\varepsilon\in(0,\varepsilon_0)$. On $B_\eta:=B_\eta(0)\subset\R^2$ we then have
    $r_\ve(t)^2-|x'|^2\geq a^2-\eta^2>0$, so $u_\ve(x',t):=\sqrt{r_\ve(t)^2-|x'|^2}$ is
    well-defined and smooth.

    It remains to check that $u_\ve$ solves \eqref{eq:ot-pde}. Differentiating
    $r_\ve=R_\ve\circ s_\ve$ and using the second equation of \eqref{eq:sys-ve} in the form
    $R_\ve'=s\,t_\ve'$ together with $\dot s_\ve=1/t_\ve'$ gives $\dot r_\ve=s$, exactly as in
    Lemma~\ref{lem:C1}; here $s_\ve$ is smooth because $t_\ve'>0$ throughout $(-\delta,\delta)$.
    Granted $\dot r_\ve=s$, the first equation of \eqref{eq:sys-ve} is the $s$-parametrized form
    of the reduced equation \eqref{eq:ode} for $r_\ve$.
    Thus $r_\ve$ satisfies \eqref{eq:ode},
    which by \eqref{eq:DT_u} is precisely $\det DT_{u_\ve}=A^{n-1}\dot B=\lambda$. Hence $u_\ve$ is
    a smooth solution of \eqref{eq:ot-pde} with $a=\tan\Theta$ and $\lambda=1/\cos^2\Theta$, and
    the assertion about $\mathcal{M}_\ve$ then follows from Lemma~\ref{lem:ot-slce}.
\end{proof}

Pass to the limit $\varepsilon\to0$. By continuous dependence of the solutions of
\eqref{eq:sys-ve} on the initial data, $(t_\ve,R_\ve)$ converges, as $\varepsilon\to0$ and
uniformly on compact subsets of $(-\delta,\delta)$, to the solution $(t_0,R_0)$ of
\eqref{eq:ode-sys} with $n=2$. By Lemma~\ref{lem:inverse}, $s\mapsto t_0(s)$ is strictly
increasing; let $s_0$ be its inverse, $r_0:=R_0\circ s_0$, and
$u_0(x',t):=\sqrt{r_0(t)^2-|x'|^2}$ on $B_\eta$. Then $r_\ve\to r_0$ and consequently
$u_\ve\to u_0$ in $C^0(B_\eta)$. Applying Lemma~\ref{lem:reg} with $n=2$ gives
$u_0\in C^{1,1/3}(B_\eta)$ while $u_0\notin C^{1,\beta}(B_\eta)$ for every $\beta>1/3$.

Finally, set $\varepsilon_j:=\varepsilon_0/2^{\,j}$ and $u_j:=u_{\varepsilon_j}$ for $j\geq1$, and
$u_\infty:=u_0$. By Lemma~\ref{lem:perturbed} each $u_j$ is a smooth graphical solution of
\eqref{eq:slce}, and the preceding paragraph shows $u_j\to u_\infty$ in $C^0(B_\eta)$ with
$u_\infty\in C^{1,1/3}(B_\eta)\setminus\bigcup_{\beta>1/3}C^{1,\beta}(B_\eta)$. Thus
$\{u_j\}_{j=1}^\infty$ and $u_\infty$ have all the properties asserted in Theorem~\ref{thm:B},
and the proof is complete.

\begin{remark}
We present a geometric perspective through the parallel surface transform.\label{rk:geometry}
For the relativistic cost~\eqref{eq:cost}, a $c$-support is the lower cap of a sphere of
radius $a$ tangent from below to the graph
$\mathcal{M}=\{(z,u(z))\}\subset\R^3$. Thus $T_u(z)$ is the horizontal projection of the
center of that sphere, or equivalently the projection of the parallel surface
\[
    \mathcal{M}_a:=\{p+a\,\nu(p):p\in\mathcal{M}\},\qquad a=\tan\Theta,
\]
where $\nu$ is the unit normal oriented toward the supporting spheres. With
$\lambda=1/\cos^2\Theta$, equation~\eqref{eq:ot-pde} says that $\mathcal{M}_a$ has constant
Gaussian curvature, while Lemma~\ref{lem:ot-slce} rewrites the same condition as
\eqref{eq:slce} on $\mathcal{M}$. In our sequence the offsets are rotationally symmetric
constant-Gauss-curvature surfaces. The limit degenerates when $r(0)=a$, so the parallel surface
collapses onto the axis at $t=0$.
\end{remark}

\begin{note}
After completion of this manuscript, the author became aware of independent related work
by Qiu and Tao~\cite{QT26}. The present paper obtains the two-dimensional
$C^{1,1/3}$ singular limit from an all-dimensional singular profile for the relativistic
heat cost. Qiu and Tao obtain the $C^{1,1/3}$ behavior through a parallel-surface
construction.
\end{note}

\end{document}